\documentclass[a4paper,12pt,leqno]{article}
\usepackage{amsthm}
\usepackage{amssymb} 
\usepackage{amsmath} 
\usepackage{enumitem}
\usepackage{titlesec}
\titlelabel{\thetitle.\quad}
\titleformat*{\section}{\large\bfseries}
\titleformat*{\subsection}{\bfseries}
\titleformat*{\subsubsection}{\bfseries}
\titleformat*{\paragraph}{\large\itseries}
\titleformat*{\subparagraph}{\itseries}
\usepackage{latexsym}
\usepackage[all]{xy}
\usepackage[top=3cm,bottom=3cm,left=2.5cm,right=2.5cm]{geometry}
\usepackage{tikz}
\usetikzlibrary{arrows,cd}  
\numberwithin{equation}{section}

\def\P{{\mathbf{P}}}
\def\Z{{\mathbb{Z}}}
\def\K{{\mathbb{K}}}
\def\A{{\mathcal{A}}}
\def\B{{\mathcal{B}}}
\def\E{{\mathcal{E}}}
\def\F{{\mathcal{F}}}

\def\cO{\mathcal{O}}
\def\cL{\mathcal{L}}

\DeclareMathOperator{\codim}{codim}

\DeclareMathOperator{\Der}{Der}

\DeclareMathOperator{\pd}{pd}

\DeclareMathOperator{\Shi}{Shi}

\newcommand{\owari}{\hfill$\square$}
\usepackage[flushmargin]{footmisc}
\usepackage{titlefoot}

\makeatletter
\def\unindentedkeywords#1{%
  \begingroup
  \renewcommand\thefootnote{} 
  \def\@makefntext##1{\noindent ##1} 
  \footnote{#1}%
  \addtocounter{footnote}{-1}%
  \endgroup
}
\makeatother

\newtheorem{theorem}{Theorem}[section]
\newtheorem{prop}[theorem]{Proposition}
\newtheorem{cor}[theorem]{Corollary}
\newtheorem{lemma}[theorem]{Lemma}

\newtheorem{conj}[theorem]{Conjecture}
\theoremstyle{definition}
\newtheorem{define}[theorem]{Definition}

\theoremstyle{remark}
\newtheorem{rem}[theorem]{Remark}

\title{On the projective dimension of some deformations of 
Weyl arrangements}

\author{Takuro Abe and Daniele Faenzi}

\date{}

\usepackage{hyperref}

\begin{document}

\unmarkedfntext{\textbf{2020 Mathematics Subject Classification}. Primary 32S22, 52S35.\smallskip\newline
}

\maketitle

\begin{abstract}
We show that the logarithmic derivation module of the cone of the
deformation $\A_{\Phi^+}^{[-j.j+2]}$ of a Weyl arrangement associated with a root system $\Phi$ of simply laced type
has projective dimension one.
In addition, we give an explicit 
minimal free resolution when  $\Phi$ is of type $A_3$ and $B_2$. Moreover, 
in the second case, we determine the jumping lines of maximal jumping order of the associated
vector bundle. 
For $j \ge 3$ and $k \ne k'$, setting $\A = c\A^{[-k,k+j]}_{B_2}$, $\A' = c\A^{[-k',k'+j]}_{B_2}$, 
this allows to distinguish $D_0(\A)$ from $D_0(\A')[4(k'-k)]$, even though these modules have the same graded Betti numbers.
\end{abstract}

\sloppy

\section{Introduction}

Let $\K$ be a field of characteristic zero, $V=\K^\ell$, $S=\mbox{Sym}^*(V^*)=
\K[x_1,\ldots,x_\ell]$ and let $\Der S:=\oplus_{i=1}^\ell S \partial_{x_i}$. An \textbf{arragement} of hyperplanes $\A$ is a finite set of linear hyperplanes in $V$. For each $H \in \A$ fix $\alpha_H \in V^*$ such that $\ker \alpha_H=H$. Let $Q(\A):=\prod_{H \in \A} \alpha_H$ be a defining polynomial of $\A$. An important algebraic invariant of $\A$ is the following.

\begin{define}
The \textbf{logarithmic derivation module} $D(\A)$ of $\A$ is
$$
D(\A):=\{\theta \in \Der S \mid \theta(\alpha_H) \in S \alpha_H,\forall H \in \A\}.
$$
We say that $\A$ is \textbf{free} with $\exp(\A)=(d_1,\ldots,d_\ell)$ if
$$
 D(\A) \cong \bigoplus_{i=1}^\ell S[-d_i], \qquad \mbox{as graded $S$-modules.}
$$
\end{define}
Free arrangements form one of the most interesting classes of configuration spaces.
However, in many cases of interest the homogeneous $S$-module $D(\A)$ is actually not free. In this paper we are particularly interested in arrangements of projective dimension one, which we consider as those that depart as little as possible from freeness.

We focus on arrangements arising from root systems, which are among the most studied kinds of arrangements.
Given a root system $\Phi$ and its set of positive roots
$\Phi^+$, the associated Weyl arrangement $\A_{\Phi^+}$ of type $\Phi$ is:
$$
\A_{\Phi^+}:=\{\alpha=0\mid \alpha \in \Phi^+\}.
$$
Historically, a lot of work has been devoted to the study of the so-called \textit{deformations of the Weyl arrangements}. This is the union of integral level sets of linear forms given by positive roots of $\Phi$, namely, given integers $a \le b \in \Z$, we set
$$
\A^{[a,b]}_{\Phi^+}=\A^{[a,b]}:=\{\alpha=k \mid \alpha \in \Phi^+, a \le k \le b,\ k \in \Z\}.
$$
These arrangements are subsets of the affine Weyl arrangement, and include the extended Catalan and Shi arrangements, see e.g. \cite{Y1}.  Among the several papers discussing the properties of deformations of Weyl and Coxeter arrangements, let us mention \cite{Ath2,Ath1,PostStan,YLinial} and \cite{Tamura}.

For an affine arrangement $\A$ defined by $Q(x_1,\ldots,x_\ell)=0$, the \textbf{coning} $c\A$ of $\A$ is an arrangement in $\K^{\ell+1}$ defined by
$$
z^{|\A|}Q\left(\frac{x_1}{z},\ldots,\frac{x_\ell}{z}\right)=0
$$
Now 
recall the following part of the conjecture in \cite{AFV}.

\begin{conj} \label{conj:AFV}
Let $\Phi$ be an irreducible root system of rank $\ell \ge 2$, let $k\ge 0$, $j \ge 2$ be integers and consider
$
\A=\A_{\Phi^+}^{[-k,k+j]}.
$
If $j=2$ and $\ell \ge 3$, or if $\ell=2$, then $\pd_S (D(c\A))=1$.
\end{conj} 

In fact the following result was proved in \cite{AFV}.

\begin{theorem}[\cite{AFV}]
Conjecture \ref{conj:AFV} is true when $\Phi$ is of type $A_2$.
\label{theoremAFV}
\end{theorem}

The first aim of this article is to show the next result, which proves a version of Conjecture \ref{conj:AFV}, valid for simply-laced root systems. In the result, the root system $\Phi$ is not necessarily irreducible, however, if it is, then $\Phi$ is of type type $ADE$. 

\begin{theorem} \label{main}
Let $k\ge 0$ and let $\Phi$ be a simply-laced root system of rank $\ell \ge 2$.
Then $\pd_S (D(c\A^{[-k,k+2]}_{\Phi^+}))=1$.
So Conjecture \ref{conj:AFV} holds for $\Phi$.
\end{theorem}

A second objective here is to give an explicit resolution of $D(\A)$ when $\A$ is a deformation of the braid arrangement having projective dimension one, in three dimensions, which is to say that $\A$ is a deformation of the Weyl arrangement of type $A_3$. 
To state it, let us first recall that for an finitely generated $S$-module $M$, one writes a minimal graded free resolution of $M$ as
\begin{equation}
\label{resolution:M}    
0 \longrightarrow
\bigoplus_{i \in \Z} S[-i-r]^{\beta_{r,i}}
\longrightarrow \cdots \longrightarrow
\bigoplus_{i \in \Z} S[-i]^{\beta_{0,i}}
\longrightarrow M
\longrightarrow 0.
\end{equation}
The integers $\beta_{i,j}$ are called the homogeneous Betti numbers of $M$. Given $i \in \Z$ such that $\beta_{0,i} \ne 0$, respectively, $\beta_{1,i} \ne 0$, we say that $M$ has $\beta_{0,i}$ generators, respectively $\beta_{1,i}$ syzygies, in degree $i$.
Here, we implicitly assume that $\beta_{r,i} \ne 0$ for some $i\in \Z$, so that $\pd_S(M)=r$.
Also, we consider the following submodule of the derivation module $D(\A)$
$$
D_0(\A):=\{\theta \in D(\A) \mid 
\theta(Q(\A))=0\}.
$$
Let
$\theta_E:=\sum_{i=1}^\ell x_i \partial_{x_i}$ be the Euler derivation.  Assuming $\A \neq \emptyset$, we have a decomposition:
$$
D(\A)=S\theta_E \oplus D_0(\A).
$$
We look at the logarithmic derivation module $D_0(\A)$ with $\A=c\A_{\Phi^+}^{[-k,k+2]}$, when $\Phi$ is of type $A_3$. This arrangement is also known as a deformation of the braid arrangement. We obtain the following result, see Theorem \ref{Aellpd1}.
\begin{theorem} \label{A3case}
Let $\Phi$ be the $A_3$ root system, $\Phi^+$ its set of positive roots and $\A = c\A_{\Phi^+}^{[-k,k+2]}$ with $k \ge 0$.
Then $D_0(\A)$ has a minimal graded free resolution of the form:
$$
0
\rightarrow 
S[-4k-8]^3
\rightarrow 
S[-4k-7]^6
\rightarrow 
D_0(\A) \rightarrow 0.
$$
\end{theorem}

As a third goal of this paper, we start the analysis of non-simply laced root systems, by computing the logarithmic  derivation module of deformed Weyl arrangements of type $B_2$.
Let $k \ge 2$ be an integer and let $r \in \{0,1\}$ be the remainder of the division of $k$ by $2$. Write $k=2m+r$ for some integer $m$.
The following complements the main result of \cite{AFV}.
Let $a \le b \in \Z$ and let
 $A_{\B_2}$ be defined by 
$xy(x^2-y^2)=0$ in $\K^2$ and let 
$$
\A^{[a,b]}_{\Phi^+}:=\{\alpha=t \mid t \in \Z,
a \le t \le b,\ \alpha \in \Phi^+_{B_2}\}.
$$

\begin{theorem} \label{resolution-B2}
Let $\Phi$ be the root system of the type $B_2$, $\Phi^+$ its set of positive roots. Then for $\A:=c\A_{\Phi^+}^{[-k,k+j]}$ with $j=2m+r$, $r \in \{0,1\}$,
$D_0(\A)$ has projective dimension one and its graded Betti numbers are:
\[
\begin{array}{c||c|c}
\mbox{degree $i$} & \mbox{generators $\beta_{0,i}$} & \mbox{syzygies $\beta_{1,i}$} \\
\hline
\hline
    1+5m+3r+4k & 1+r& \\
    \hline
    2+5m+3r+4k & 2 & 1+r\\
    \hline
    3+5m+3r+4k & 2 & 2\\    
    \hline
     \vdots & \vdots & \vdots \\
     \hline
    1+6m+3r+4k & 2 & 2
\end{array}
\]
\end{theorem}

Finally, we analyze the vector bundle $\E$ on $\P^2=\P^2_{\K}$ obtained by sheafifying $D_0(\A)$, with $\A = c\A^{[-k,k+j]}_{\Phi^+}$.
We set $\F=\E(4k+2j+2)$ and we note that $c_1(\F)=0$. 
For any integer $b > 0$, a $b$-jumping line of $\E$ is a line $\ell \subset \P$ such that $\F|_\ell \simeq \cO_\ell(-b) \oplus \cO_\ell(b)$.
Given a second integer $k'$,  we denote by $\E'$ the sheafification of $D_0(\A')$, with $\A' = c\A^{[-k',k'+j]}_{\Phi^+}$.
From Theorem \ref{resolution-B2} one computes the Chern classes of $\F$ and finds that $\F$ is a slope-stable bundle in the sense of Mumford-Takemoto for $j\ge 3$, so that a generic line is not jumping. It turns out that one can determine the $a$-jumping lines of $\E$, when $a$ reaches its maximum value, which is $j-1$, see Theorem \ref{B2:jump}. 
From this result we may see that, when $k \ne k'$, the two $S$-modules $D_0(\A)$ and $D_0(\A')[4(k'-k)]$ are not isomorphic when $k \ne k'$, even though they share the same graded Betti numbers, cf. Corollary \ref{not-isomorphic}. We summarize this in the next result.

\begin{theorem} \label{summary-B2}
Let $j = 2m+r \ge 3$, $r\in \{0,1\}$. The rank-2  bundle $\E$ is stable, with 
\[
c_1(\F)=0, \qquad c_2(\F)=2 m^2 + 2 m r + r -1.
\]
For $a > j-1$, there is no $a$-jumping line for $\E$ while  $L:\{y=(k+j)z\}$ and $H:\{y=(-k-1)z\}$ are the only $(j-1)$-jumping lines of $\E$.
For $k \ne k'$, $D_0(\A)$ and $D_0(\A')[4(k'-k)]$ have the same graded Betti numbers but are not isomorphic $S$-modules.
\end{theorem}

The organization of this article is as follows. In \S \ref{section:preliminaries} we introduce definitions and results mainly about arrangements of projective dimension zero (aka free) and one as well as some tools to analyze the behavior of resolutions of length one upon adding or removing hyperplanes. In \S \ref{section:simply-laced} we prove Theorem \ref{main}, see Theorem \ref{Aellpd1}. In \S \ref{section:A3} we give an explicit minimal free resolution of 
$D(\A^{[-j,j+2]}_{\Phi^+})$ when the root system is of type $A_3$, see Theorem \ref{k2case}, which shows Theorem \ref{A3case}. In \S \ref{section:B2} we look at deformed Weyl arrangements of type $B_2$ and show Theorem \ref{resolution-B2} and Theorem \ref{summary-B2}, where this last result follows from Theorem \ref{B2:jump} and Corollary \ref{not-isomorphic}.
\bigskip

\noindent{\textbf{Acknowledgments. }}
T. A. is partially supported by JSPS KAKENHI Grant Numbers JP23K17298 and JP25H00399.
D. F. partially supported by FanoHK ANR-20-CE40-0023, EIPHI ANR-17-EURE-0002, Bridges ANR-21-CE40-0017. 

\bigskip

\section{Background material and basic results}

\label{section:preliminaries}

In this section, we first collect some background material and fix the notation alongside with some conventions. 
Then, we show some basic theorems regarding arrangements of low projective dimension. These will be crucial in the proofs  of our main results concerning resolutions of arrangements of Weyl arrangements.

\subsection{Basic notions}

For an arrangement $\A$ in $V=\K^\ell$, let 
$$
L(\A):=\{\cap_{H \in \B} H \mid \B \subset \A\}
$$
be the \textbf{intersection lattice} of $\A$. For $X \in L(\A)$, define $\mu(X)$ by 
$\mu(V)=1$, and by 
$$
\mu(X):=-\sum_{Y \in L(\A),\ X \subsetneq Y} \mu(Y)
$$
for $X \neq V$. Then $\mu:L(\A) \rightarrow \Z$ is called the \textbf{M\"obius function}. Let 
$$
\chi(\A;t):=\sum_{X \in L(\A) } \mu(X)t^{\dim X}
$$
be the \textbf{characteristic polynomial} of $\A$. When $\A \neq \emptyset$, it is known that 
$\chi(\A;t)$ is divisible by $t-1$. Let $$
\chi_0(\A;t):=\chi(\A;t)/(t-1).$$

Let $$
D(\A):=\{\theta \in \Der S \mid \theta(\alpha_H) \in S\alpha_H\ 
(\forall H \in \A)\}.
$$
We say that $\A$ is \textbf{free} with $\exp(\A)=(d_1,\ldots,d_\ell)$ if 
$$
D(\A) \simeq \oplus_{i=1}^\ell S[-d_i].
$$
If $\A $ is not empty, there is a direct sum decomposition 
$$
D(\A)=S\theta_E \oplus D_0(\A),
$$
where 
$$
D_0(\A)=\{ \theta \in D(\A) \mid \theta(Q(\A))=0\}.
$$
So for a non-empty free arrangement $\A$, its exponents contain $1$. 
A very fundamental properties of a free arrangement are as follows. 

\begin{theorem}[\cite{Sa}]
Let $\theta_1,\ldots,\theta_\ell \in D(\A)$ be homogeneous 
derivations. Then $\theta_1,\ldots,\theta_\ell$ form a basis for $D(\A)$ if and only if $\theta_1,\ldots,\theta_\ell$ are independent over $S$, and 
$$
|\A|=\sum_{i=1}^\ell \deg \theta_i.
$$
\label{saito}
\end{theorem}

\begin{prop}[Theorem 4.37, \cite{OT}]
If $\A$ is free, then for all $X \in L(\A)$, considering the \textbf{localization} of $\A$ at $X$, namely
$$
\A_X:=\{H \in \A \mid X \subset H\},
$$
we have that
$\A_X$ is free.
\label{local}
\end{prop}

The next result is very useful to determine the freeness of the arrangement which can be obtained by deleting a hyperplane from a free arrangement, see Theorem 4.49 in \cite{OT} for the proof.

\begin{theorem}[Terao's deletion theorem, \cite{T1}]
Let $H \in \A$ and $\A$ be free with $\exp(\A)=(d_1,\ldots,d_\ell)$. If $\A^H:=\{L \cap H \mid 
L \in \A \setminus \{H\}\}$ is free with $\exp(\A^H)=(d_1,\ldots,d_{i-1},\hat{d_i},d_{i+1},\ldots,d_\ell)$, then $\A \setminus \{H\}$ is free with $\exp(\A\setminus \{H\})=(d_1,\ldots,d_{i-1},d_i-1,d_{i+1},\ldots,d_\ell)$.
\label{Teraodeletion}
\end{theorem}

\begin{theorem}[Multiple deletion theorem, Theorem 1.3,  \cite{AT2}]
Let $\A$ be free with $\exp(\A)=(1,d_2,\ldots,d_\ell)$ with $1<d_2 
\le d_3 \le \cdots \le d_\ell$. 
If $H \in \A$ satisfies $|\A|-|\A^H|=d_2$, then $\A':=\A 
\setminus \{H\}$ is free with $\exp(\A')=
(1,d_2-1,d_3,\ldots,d_\ell)$.

\label{MDT}
\end{theorem}

Let us recall a fundamental polynomial to analyze the addition and deletion of arrangements.

\begin{define}[\cite{T1}]
For $H \in \A$, \textbf{Terao's polynomial $B$}, $B=B(\A,H)$, is defined as follows. 
Let $\nu:\A^H \rightarrow \A':=\A \setminus \{H\}$ be a section of the restriction $\A' \ni L \mapsto H \cap L \in \A^H$. Then $\B$ is defined by 
$$
B=\displaystyle \frac{Q(\A')}{\prod_{X \in \A^H} \alpha_{\nu(X)}}.
$$
\label{B}
\end{define}

The definition of $B$ depends on the choice of sections, however, it is easy to show that $B$ is unique up to non-zero scalar after taking modulo $\alpha_H$. 

\begin{prop}[$B$-sequence, Theorem 1.6, \cite{A12}]
Let $H \in \A$ and $\A':=\A \setminus \{H\}$ and 
$B:=B(\A,H)$. 
 Then for a map $\partial:D(\A') \rightarrow \overline{SB}$ defined by 
 $$
 \partial(\theta):=\overline{\theta(\alpha_H)},
 $$
 there is an exact sequence called a \textbf{$\B$-sequence}:
 $$
 0 \rightarrow D(\A) \rightarrow D(\A') \stackrel{\partial}{\rightarrow} \overline{SB}.
 $$
 Here $D(\A) \rightarrow D(\A') $ is the natural inclusion. In addition, we have the following exact sequence:
 $$
 0 \rightarrow D_0(\A) \stackrel{j}{\rightarrow} D_0(\A') \stackrel{\partial}{\rightarrow} \overline{SB}
 $$
 up to isomorphisms.
 \end{prop}

 \begin{proof}
 For the first sequence , we refer to \cite[Theorem 1.6]{A12}. For the second one, we need to define a new map, which we denote by $j$, as in fact there is no natural inclusion 
 $D_0(\A) \subset D_0(\A')$. 
Let $Q(\A)=Q, Q(\A')=Q'$, so $\alpha_H Q'=Q$. Let $\theta \in D_0(\A)$. Then 
$$
\theta(Q)=\alpha_H\theta(Q')+Q'\theta(\alpha_H)=0.
$$
Since $\theta_E(Q')=|\A'|Q'$, we have
$$
j(\theta):=\theta+\frac{\theta(\alpha_H)}{|\A'|\alpha_H}\theta_E \in 
D_0(\A').
$$
Now 
$$
\partial \circ j(\theta)=
\overline{\theta(\alpha_H)+\frac{\theta(\alpha_H)}{|\A'|\alpha_H}\alpha_H}=0
$$
since $\theta(\alpha_H) \in S \alpha_H$. Conversely, 
assume that $\partial(\varphi)=0$ for $\varphi \in D_0(\A')$, 
i.e., $\varphi(\alpha_H)$ is divisible by $\alpha_H$. Then $\varphi \in D(\A)$. Thus 
$$
\varphi(Q)=Q' \varphi(\alpha_H).
$$
Hence it is easy to show that 
$$
s(\varphi):=\varphi-\frac{\varphi(\alpha_H)}{|\A|\alpha_H} \theta_E \in D_0(\A),
$$
and we can check that $j\circ s (\varphi)=\varphi$, which shows that 
$\mbox{ker} (\partial)=\mbox{Im}(j)$.
\end{proof}

It is well-known that $D(\A)$ and $D_0(\A)$ are both reflexive. So \cite[Corollary 1.13]{Eis} shows the following.

\begin{prop} We have a natural isomorphism
$$
H^0_*(\widetilde{D(\A))}:=\bigoplus_{d \in \Z} H^0(\widetilde{D(\A)} (d))
 \simeq D(\A).
$$    
This restricts to a natural isomorphism
$$
H^0_*(\widetilde{D_0(\A))}
\simeq D_0(\A).
$$    
\label{Eis}
\end{prop}

\subsection{Arrangements of projective dimension zero or one}

Let us collect here some results about arrangements whose module of
logarithmic derivations has small projective resolutions, mainly
projective dimension zero (free arragements) or projective dimensione one. For that let us introduce some definitions.

\begin{define}
We say that $\A$ is \textbf{locally free} if $\A_X$ is free for all 
$X \in L(\A)$ with $X \neq \cap_{H \in \A}H$. Also, we say that $\A$ is \textbf{locally free along $H$} if $\A_X$ is free for all 
$X \in L(\A^H)$ with $X \neq \cap_{H \in \A}H$. Moreover, we say that $\A$ is \textbf{locally free along $H$ in codimension $k$} if 
$\A_X$ is free for all 
$X \in L(\A^H)$ with $X \neq \cap_{H \in \A}H$ and $\codim_V X=k$.
\label{locallyfree}
\end{define}

\begin{theorem}[Theorem 3.2, \cite{A12}]
Let $H \in \A$ and $\A':=\A \setminus \{H\}$. 
Assume that $\A$ is locally free along $H$ in codimension three. Then 
$\partial$ in Proposition \ref{B} is surjective if $\pd_S(D(\A)) \le 1$.
\label{FST}
\end{theorem}

Now let us recall Yoshinaga's celebrated criterion for freeness. 
For $H \in \A$ let $\A^H:=\{H \cap L \mid L \in \A \setminus \{H\}\}$ be an arrangement in $\K^{\ell-1} \simeq H$, and 
let $m^H:\A^H \rightarrow \Z_{>0}$ be defined by 
$$
m^H(X):=|\{L \in \A \setminus \{H\} \mid L \cap H=X\}|.
$$
Then the pair $(\A^H,m^H)$ is called the \textbf{Ziegler restriction} of $\A$ onto $H$. We can define the logarithmic derivation module $D(\A^H,m^H)$ of $(\A^H,m^H)$ onto $H$ by 
$$
D(\A^H,m^H):=\{ \theta \in \Der \overline{S} \mid \theta(\alpha_X) \in \overline{S} \alpha_X^{m^H(X)}\ 
(\forall X \in \A^H)\}.
$$
Let us recall a fundamental result on the Ziegler restriction.

\begin{theorem}[\cite{Z}]
If $\A$ is free with $\exp(\A)=(1,d_2,\ldots,d_\ell)$, then 
$(\A^H,m^H)$ is free with exponents $(d_2,\ldots,d_\ell)$. Moreover, we can choose a basis $\theta_E,\theta_2,\ldots,\theta_\ell$ for $D(\A)$ such that $\overline{\theta_2},\ldots,\overline{\theta_\ell}$ form a basis for $D(\A^H,m^H)$, where $\overline{\theta_i}$ denotes $\theta_i$ modulo $\alpha_H$.
\label{Z}
\end{theorem}

Here $\overline{S}:=S/\alpha_H S$. It is known that all Ziegler restrictions of three dimensional arrangements are free, so the degrees of basis for it is called the \textbf{exponents} of $(\A^H,m^H)$, denoted by $\exp(\A^H,m^H)$. This gives the following criterion for the freeness of $\A$.

\begin{theorem}[Theorem 3.2, \cite{Y2}]
Let $\A$ be an arrangement in $\K^3$ and let $H \in \A$. Assume that 
$\exp(\A^H,m^H)=(d_1,d_2)$. Then $\A$ is free with exponents $(1,d_1,d_2)$ if and only if 
$\chi_0(\A;0)=d_1d_2$.
\label{Ycriterion}
\end{theorem}

Also by using multiarrangement theory, we can show the following.

\begin{theorem}
Let $\A$ be free with $\exp(\A)=(1,d_2,\ldots,d_\ell)$, $1<d_2 <
d_3 \le d_4 \le \cdots \le d_\ell$. 
If $H \in \A$ satisfies $|\A|-|\A^H|=d_3$, then $\A':=\A 
\setminus \{H\}$ is free with $\exp(\A')=
(1,d_2,d_3-1,d_4,\ldots,d_\ell)$.
\label{MDT2}
\end{theorem}

\noindent
\textbf{Proof}.
Let $\theta_E,\theta_2,\ldots,\theta_\ell$ be a basis for $D(\A)$ with $\deg \theta_i=d_i\ (i=2,\ldots,d_\ell)$. By Theorem \ref{Z}, we can choose $\theta_2,\ldots,\theta_\ell$ in such a way that, denoting by $\overline{\theta_i}$ the class of $\theta_i$ modulo $\alpha_H$, then $\overline{\theta_2},\ldots,\overline{\theta_\ell}$ form a basis for $D(\A^H,m^H)$. Note that $\varphi:=
\frac{Q(\A^H,m^H)}{Q(\A^H)} \overline{\theta_E}$ is in $D(\A^H,m^H)$ by definition, and that $\deg \varphi=1+|\A|-1-|\A^H|=d_3=\deg \theta_3$ by the assumption. 
So we can express $\varphi$ is the form 
$$
\varphi=f \theta_2+c \theta_3
$$
for some $f \in S$ and $c \in \K$. Assume that $c=0$. Then $$
\frac{Q(\A^H,m^H)}{Q(\A^H)}\overline{\theta_E}=
f \overline{\theta_2},
$$
showing that $\overline{\theta_2}=g \overline{\theta_E}$ for some $g \in S$ with $g \mid Q(\A^H,m^H)/Q(\A^H,m^H)$. 
For $g \overline{\theta_E}$ to be in $D(\A^H,m^H)$, we must have $(Q(\A^H,m^H)/Q(\A^H) \mid g$, contradicting $d_2 < d_3$. So $c \neq 0$. Hence we can replace $\theta_3$ by $\varphi$ to show that 
$\overline{\theta_2},\varphi,\overline{\theta_4},\ldots,
\overline{\theta_\ell}$
form a basis for $D(\A^H,m^H)$. Then by dividing $\varphi$ by 
$Q'(\A^H,m^H)/Q(\A^H)$, we know that $\overline{\theta_E},
\overline{\theta_2},\overline{\theta_4},\ldots,\overline{\theta_\ell}$ form a basis for $D(\A^H)$ by Theorem \ref{saito}. In particular, 
$\A^H$ is free 
with $\exp(\A^H)=(1,d_2,d_4,\ldots,d_\ell)$. Thus Theorem \ref{Teraodeletion} shows that $\A\setminus \{H\}$ is free 
with exponents $(1,d_2,d_3-1,d_4,\ldots,d_\ell)$. \owari
\medskip

Let us introduce a fundamental tool to analyze deletion of hyperplanes of a free arrangement.

\begin{define}[Definition 1.1, \cite{A5}] 
We say that $\A$ is \textbf{strictly plus-one generated (SPOG)} with $\exp(\A)=(d_1,\ldots,d_\ell)$ and \textbf{level} $d$ if there is a minimal free resolution 
$$
0 \rightarrow S[-d-1] \rightarrow \bigoplus_{i=1}^\ell S[-d_i]\oplus S[-d] \to D(\A) \to 0.
$$
So $\beta_{1,d}=1$ and otherwise $\beta_{k,i}=0$ except when $k=0$ and $i \in \{d_1,\ldots,d_\ell,d\}$. 
\end{define}

\begin{theorem}[Theorem 1.4, \cite{A5}]
Let $\A$ be free with $\exp(\A)=(1,d_2,\ldots,d_\ell)$ and $H \in \A$. If $\A':=\A \setminus \{H\}$ is not free, then there is a minimal free resolution 
$$
0 \rightarrow S[-d-1]\rightarrow 
\oplus_{i=2}^\ell S[-d_i] \oplus S[-d] \rightarrow D_0(\A) 
\rightarrow 0,
$$
where $d:=|\A|-|\A^H|-1= \deg B$. 
 Moreover, for a derivation $\varphi \in 
D(\A') \setminus D(\A)$ of degree $d$ in a minimal set of generators
for $D(\A')$, it holds that $\partial(\varphi)=\overline{B}$ up to non-zero scalar.
\label{SPOG}
\end{theorem}

Let us recall some free arrangements which appear frequently in the rest of this paper.

\begin{theorem}[Theorem 1.2, \cite{Y1}] \label{thm:CatShi}
Let $\Phi$ be a root system, $\Phi^+$ its set of positive roots. Let $k \ge 0$ and $h$ the Coxeter number of $\Phi$. Also let $\exp(\A_{\Phi^+})=(d_1,\ldots,d_\ell)$.
\begin{enumerate}[label=(\roman*)]
    \item \label{shi-def}
$\A:=c\A_{\Phi^+}^{[-k,k+1]}$ is free with exponents 
$(1,kh,\ldots,kh)$.
\item \label{catalan-def}
$\B:=c\A_{\Phi^+}^{[-k,k]}$ is free with exponents 
$(1,kh+d_1,,\ldots,kh+d_\ell)$.
\end{enumerate}
\end{theorem}

The arrangement appearing in \ref{shi-def} above is called an \textbf{extended Shi arrangement}, while the one described in \ref{catalan-def} is called an \textbf{extended 
Catalan arrangement}.

\subsection{Resolutions and addition/deletion}

Let us introduce two fundamental results to determine the resolution of a line arrangement. In the next lemma, we rewrite resolutions in a slightly different form from \eqref{resolution:M}.

\begin{lemma}
Let $\A'$ be a line arrangement in $\P^2$, $H$ be a hyperplane not in $\A'$ and set $\A=\A' \cup \{H\}$. Let 
$$
0 \rightarrow F_2 \rightarrow F_1 \rightarrow D_0(\A) \rightarrow 0
$$
be a minimal graded free resolution of $D_0(\A)$ with 
$$
F_i=\bigoplus_{j=1}^{n_i} S[-d_{ij}].
$$
If $d:=|\A|-1-|\A^H| >\max_{1 \le j \le n_{1}} d_{1j}$, and
$\max_{1 \le j \le d_{n_2}} d_{2j} <d+3$, then we have a minimal free
 resolution
 $$
0 \rightarrow F_2 \oplus S[-d-1] \rightarrow F_1
\oplus S[-d]\rightarrow D_0(\A') \rightarrow 0.
$$
\label{lemma1}
\end{lemma}

\begin{proof}
Let $(\theta_1,\ldots,\theta_{n_1})$ be a set of minimal generators for $D_0(\A)$ with $\deg \theta_i=d_i$, for $1 \le i \le n_1$. Consider the $B$-sequence in Proposition \ref{B}:
\begin{equation}
    \label{B-sequence}
0 \rightarrow D_0(\A) \rightarrow D_0(\A') 
\rightarrow \overline{S}[-d].
\end{equation}

Sheafifying the above sequence with an exact sequence of coherent sheaves of $\mathcal{O}_{\P^2}$-modules
$$
0 \rightarrow \widetilde{D_0(\A)} \rightarrow \widetilde{D_0(\A') }
\rightarrow \mathcal{O}_H(-d) \rightarrow 0
$$

This sequence right-exact since a line arrangement $\A$ in $\P^2$ is locally free. So, to show that \eqref{B-sequence} is right-exact,  by using Proposition \ref{Eis}, it suffices to show that 
\[
H^1(\widetilde{D_0(\A)}(d))=0.
\]
However, this follows by the minimal free resolution and the condition on $d_{2j}$. 
Therefore, we have a set of elements $D_0(\A')$, namely
$$
(\theta_1,\ldots,\theta_{n_1}, \varphi),
$$
where $\varphi $ maps to the a generator of $S[-d]$, so that $\deg \varphi=d$. 

Since $(\theta_1,\ldots,\theta_{n_1})$ is a system of generators of $D_0(\A)$, and $\varphi$ generates $S[-d]$, $(\theta_1,\ldots,\theta_{n_1}, \varphi)$ is a system of generators of $D_0(\A')$, owing to the right exactness of \eqref{B-sequence}. 
Let us check that these generators are minimal.

Clearly, we cannot remove $\varphi$ otherwise the sequence cannot \eqref{B-sequence} be right exact. If we can remove, say $\theta_1$, then by the assumption on $d$ and $d_{1j}$, $\theta_1$ can be expressed as a linear combination of $\theta_2,\ldots,\theta_{n_1}$, contradicting their minimality as a generator for $D_0(\A')$. So these form a minimal set of generators for 
$D_0(\A')$. Looking at the relations among these generators, since $\alpha_H \varphi \in D_0(\A)$, there is a new relation among them at degree $d+1$, and the syzygy corresponding to this relation cannot be removed. Also, we cannot remove the other syzygies since would contradict the minimality of the generators of $D_0(\A)$.
\end{proof}

\begin{lemma}
Let $\A'$ be a line arrangement in $\P^2$, let $H$ be a hyperplane not in $\A'$ and set $\A=\A' \cup \{H\}$. Let 
$$
0 \rightarrow F_2 \rightarrow F_1 \rightarrow D_0(\A') \rightarrow 0
$$
be a minimal graded free resolution with 
$$
F_i=\bigoplus_{j=1}^{n_i} S[-d_{ij}].
$$
If $d:=|\A^H|-1 >\max_{1 \le j \le n_{1}} d_{1j}$, and
$\max_{1 \le j \le d_{n_2}} d_{2j} <d+2$, then we have a minimal free
 resolution
 $$
0 \rightarrow F_2[-1] \oplus S[-d-1] \rightarrow F_1[-1]
\oplus S[-d]\rightarrow D_0(\A) \rightarrow 0.
$$
\label{lemma2}
\end{lemma}

\begin{proof}Let $(\theta_1,\ldots,\theta_{n_1})$ be a minimal system of generators for $D_0(\A')$ with $\deg \theta_i=d_i$, for $1 \le i \le n_1$. Now consider the Euler sequence:
\begin{equation} \label{euler}
0 \rightarrow D_0(\A')[-1] \rightarrow D_0(\A') 
\rightarrow D_0(\A^H).
\end{equation}

Note that $\A^H$, being a two-dimensional central arrangement, is free with $\exp(\A^H)=(1,d)$.
Also, the sheafified sequence 
$$
0 \rightarrow \widetilde{D_0(\A')}(-1) \rightarrow \widetilde{D_0(\A) }
\rightarrow  \mathcal{O}_H(-d) \rightarrow 0
$$
is right-exact since a line arrangement in $\P^2$ is locally free. So, to show that \eqref{euler} is right-exact, again it suffices to show that 
\[H^1(\widetilde{D_0(\A')}(d-1))=0.\]
This follows by the minimal free resolution and the condition on $d_{2j}$. So, by the exactness of \eqref{euler}, we have a system of generators of $D_0(\A)$:
$$
(\alpha_H \theta_1, \ldots,\alpha_H \theta_{n_1}, \varphi),
$$
where $\varphi$ maps to a generator of the non-Euler basis element of $D_0(\A^H)$, thus $\deg \varphi=d$.

Let us show that these generators are minimal. Clearly we cannot remove $\varphi$, otherwise \eqref{euler} cannot be exact on the right. If we can remove, say $\alpha_H\theta_1$, then by the assumption on $d$ and $d_{1j}$, $\alpha_H\theta_1$ can be expressed as a linear combination of $\alpha_H \theta_2,\ldots,\alpha_H \theta_{n_1}$. Namely, if 
$\varphi$ occurs in the expression of $\alpha_H\theta_1$, then the image of $\varphi$ cannot be a basis for $D_0(\A^H)$. 
So we have a relation
$$
\alpha_H \theta_1=\sum_{i=2}^{n_1} f_i\alpha_H  \theta_i.
$$
By dividing by $\alpha_H$, we know that $\theta_2 \in \langle 
\theta_3,\ldots,\theta_{n_1}\rangle$, 
contradicting the minimality of our generators of $D_0(\A')$. So they form a minimal set of generators for 
$D_0(\A)$. About the relations among these generators, since $\alpha_H \varphi$ maps to zero in $ D_0(\A^H)$, there is a new relation among these generators, in degree $d+1$. Since this is the only relation including $\varphi$, we cannot remove this syzygy. Also, we cannot remove the other syzygies since this would contradicts the minimality of the generators for $D_0(\A')$.
\end{proof}

\subsection{A multiple deletion theorem for SPOG arrangements}

To prove Theorem \ref{A3case}, let us show the following, which is a SPOG-version of the multiple deletion 
theorem in \cite{AT2}.

\begin{theorem} \label{multiple-deletion}
Let $\A$ be a free arrangement with $\exp(\A)=(1,d_2,\ldots,d_\ell)$ and 
let $H_1,\ldots,H_p \in \A$ be distinct hyperplanes. 
Let $e_i:=|\A|-|\A^{H_i}|-1$ for $i=1,\ldots,p$. If $|\A_{H_i \cap H_j}|=2 $ all $1 \le i <j \le p$, then  for any $H \in \A':= 
\A \setminus \{H_1,\ldots,H_p\}$, $D(\A')$ has the following 
free resolution:
$$
0 \rightarrow 
\bigoplus_{i=1}^p S[-e_i-1] \rightarrow 
\bigoplus_{i=1}^\ell S[-d_i] \oplus \bigoplus_{i=1}^p S[-e_i] 
\rightarrow D(\A') \rightarrow 0.
$$
Moreover, if $\A \setminus \{H_i\}$ is not free for all $1 \le i \le p$, then this is a minimal free 
resolution. 
\label{SPOGMDT}
\end{theorem}

\begin{proof}
    For $1 \le i \le p$, let $\A_i:=\A \setminus \{H_i\}$. 
Theorem \ref{SPOG} shows that there is $\varphi_i \in D(\A_i)$, for all $i\in \{1,\ldots,p\}$ such that $\varphi_i(\alpha_{H_i}) =B_i $ modulo $\alpha_{H_i}$, where $B_i$ is Terao's polynomial $B$ with respect to $H_i$, thus $\deg B_i=\deg \varphi_i=e_i$.

Note that Theorem \ref{SPOG} says that 
$D(\A_1)=\langle 
\theta_E,\theta_2,\ldots,\theta_\ell,\varphi_1\rangle_S$, where $\theta_E,
\theta_2,\ldots,\theta_\ell$ form 
a basis for $D(\A)$ with $\deg \theta_i=d_i$ for $2 \le i \le \ell$.
Moreover, we have $D(\A_1) \ni \alpha_{H_i} \varphi_i$ for $i \ge 2$. Since 
$|\A_{H_i \cap H_j}|=2$, we compute, for all  $i \ge 1$, 
$$
|\A \setminus \{H_1,\ldots,H_{i+1}\}| -|(\A \setminus \{H_1,\ldots,H_i\})^{H_{i+1}}|=e_{i+1}=\deg B_i=\deg \varphi_{i+1}(\alpha_{H_{i+1}}).
$$
This allows to show, by induction on $i \ge 1$, that
\begin{equation} \label{free Ai}
    D(\A_i)=\langle 
\theta_E,\theta_2,\ldots,\theta_\ell,\varphi_1,\ldots,\varphi_i\rangle_S.
\end{equation}

Indeed, first note that, for $i=1$, this is nothing but Theorem \ref{SPOG}.
Next, assume that \eqref{free Ai} holds up to $i-1 \ge 1$. 
Let $\theta \in D(\A \setminus \{H_1,\ldots,H_i\})$.
Since $\varphi_i(\alpha_{H_i})=B_i$, there is $f \in S$ such that 
$\theta-f\varphi_i$ is tangent to $H_i$, so that 
$\theta-f\varphi_i \in D(\A_{i-1})=\langle 
\theta_1,\ldots,\theta_\ell,\varphi_1,\ldots,\varphi_{i-1}\rangle_S$. This shows that the 
$\theta_E,\theta_2,\ldots,\theta_\ell,\varphi_1,\ldots,\varphi_i
$ generate the $S$-module $D(\A_i)$.
Since $\varphi_i \in D(\A \setminus \{H_i\})$ by the definition, 
\begin{equation}
D(\A)
\ni
\alpha_{H_i}\varphi_i=\sum_{j=1}^\ell f_{ij} \theta_j
\label{eq1}
\end{equation}
for any $i$ by some $f_{ij} \in S$. Now let 
$$
\sum_{i=1}^\ell g_i \theta_i+\sum_{i=1}^p h_i \varphi_i=0.
$$
Note that $\theta_i \in D(\A)$ and $\varphi_i \in D(\A \setminus \{H_i\})$ for all $i$. So $h_i \varphi_i \alpha_{H_i} \in S(\alpha_{H_i})$, implying that $\alpha_{H_i} \mid h_i$ for all $i$. So we can re-write (\ref{eq1}) as
$$
\sum_{i=1}^\ell g_i \theta_i+\sum_{i=1}^p h_i (\alpha_{H_i}\varphi_i)=0.
$$
By (\ref{eq1}), this can be expressed as 
$$
\sum_{i=1}^\ell g_i \theta_i+\sum_{i=1}^p \sum_{j=1}^\ell h_if_{ij} \theta_j=\sum_{i=1}^\ell (g_i+h_i\sum_{j=1}^\ell f_{ij})\theta_i=0.
$$
Since $\theta_1,\ldots,\theta_\ell$ are $S$-independent, 
$g_i+h_i\sum_{j=1}^\ell f_{ij}=0$. Hence, any relation among $\theta_1,\ldots,\theta_\ell,\varphi_1,\ldots,\varphi_p$ can be expressed by the relations (\ref{eq1}). Therefore, 
we get the desired free resolution, since in each of the above relations, $\varphi_i$ appears only once.

Let us show that this is minimal, i.e., 
$\theta_E,\theta_2,\ldots,\theta_\ell,\varphi_1,\ldots,\varphi_p$ form a minimal set of generators for $D(\A')$ when $\A \setminus \{H_i\}$ is not free for all $i \in \{1,\ldots,p\}$. 
Assume this is not the case and 
first let 
\begin{equation}
\theta_\ell=\sum_{i=1}^{\ell-1} f_i \theta_i+\sum_{i=1}^p g_i \varphi_i,
\label{eq2}
\end{equation}
for some $f_1,\ldots,f_{\ell-1}$ and $g_1,\ldots,g_p \in S$.
Since $\theta_i$'s are all in $D(\A)$ but $\varphi_i \in D(\A \setminus \{H_i\})$, we see  
that $\alpha_{H_i} \mid g_i$ for all $i$. Let us rewrite (\ref{eq2}) as 
\begin{equation}
\theta_\ell=\sum_{i=1}^{\ell-1} f_i \theta_i+\sum_{i=1}^p g_i 
(\alpha_{H_i}\varphi_i).
\label{eq3}
\end{equation}
So \eqref{eq3}
is in $D(\A)$.
Since $\theta_1,\ldots,\theta_\ell$ form a basis for $D(\A)$, we have $\deg f_i >0$ or $f_i=0$. Also, since $\alpha_{H_i} \varphi_i$ is a linear combination of $\theta_1,\ldots,\theta_\ell$, we must have $g_i=1$, up to a non-zero scalar, for some $i \in \{1,\ldots,p\}.$ We may put $i=1$. Then replacing $\theta_\ell$ by $\theta_\ell-\sum_{i=1}^{\ell-1} f_i \theta_i-\sum_{i=2}^p g_i (\alpha_{H_i}\varphi_i)$, we may assume that 
$\theta_\ell=\alpha_1 \varphi_1$. Then 
$D(\A \setminus \{H_1\})$ is free, with basis 
$$
\theta_1/\alpha_{H_1},\theta_2,\ldots,\theta_\ell.
$$ 
However, this is a contradiction. 

Next, assume that $$
\varphi_p=\sum_{i=1}^{\ell} f_i \theta_i+\sum_{i=1}^{p-1} g_i \varphi_i,
$$
for some $f_1,\ldots,f_{\ell}$ and $g_1,\ldots,g_{p-1} \in S$.
Since the right hand side is a derivation which is tangent to $H_p$, the derivation appearing on the left 
hand side is also tangent to $H_p$. This is again a contradiction. Thus, the set of generators we constructed is indeed minimal.
\end{proof}

\subsection{Splitting types of logarithmic bundles}

Finally let us recall the following result which deals with 
splitting types of logarithmic bundles.
Let $\A$ be an arrangement in $\K^3$ and consider the sheafification $\E:=
\widetilde{D_0(\A)}$. This is a vector bundle on $\P^2=\P^2_\K$.
We take this result from \cite{AFV}, following an argument developed in \cite{faenzi-valles}.

\begin{theorem}[Theorem 3, \cite{AFV}] \label{AFV}
Let $H \in \A$ and $L$ be a line of $\P^2$ not in $\A$. 
Let $\A \cap L:=\{L \cap K \mid K \in \A\}$. 
\begin{enumerate}[label=\roman*)]
    \item
If $|\A|-|\A^H|\le \lceil \frac{|\A|-|\A^H|}{2} \rceil$, then  
$
E|_H \simeq \mathcal{O}_H(-|\A|+|\A^H|) \oplus \mathcal{O}_H(|\A^H|-1).
$
\item
If $|\A|-|\A \cap L| \ge \lceil \frac{|\A|}{2}\rceil$, then 
$
E|_L\simeq \mathcal{O}_L(-|\A|+|\A\cap L|) \oplus 
\mathcal{O}_L(-|\A \cap L|+1).
$
\end{enumerate}
\end{theorem}

\section{Simply laced root systems}

\label{section:simply-laced}

Now let us show the first main result of this article, namely Theorem \ref{main}.
For this, let us start with the following elementary result.

\begin{theorem}
Assume that $\pd_S (D(\A)) = 1$ and $\A$ is locally free along $H$ in codimension three. Then $\pd_S (D(\A')) = 1$ for 
$\A'=\A \setminus \{H\}$. 
\label{pd1}
\end{theorem}

\begin{proof}
By Theorem \ref{FST}, we have the following exact $B$-sequence from Proposition \ref{B}:
$$
0 \rightarrow D(\A) 
\rightarrow D(\A') \stackrel{\partial}{\rightarrow }
\overline{SB} \rightarrow 0.
$$
Now apply the Ext-long exact sequence combined with $\pd_S (\overline{SB})=1$ to show that $\pd_S (D(\A')) \le 1$.\end{proof}

Our strategy is to show that $\pd_S(D(\A)) \le 1$ for $\A:=
c\A^{[-k,k+2]}_{\Phi^+}$ when $\Phi$ is a simply-laced root system by deleting hyperplanes from the free extended Shi arrangement $\B:=\Shi^{[-k-1,k+2]}$ as in Theorem \ref{thm:CatShi}.
For this, it suffices to show that at each step of the deletion process, the arrangement is locally free along $H$ in codimension three.

Write $\Phi^+=\{\alpha_1,\ldots,\alpha_n\}$ with $
\mbox{ht}(\alpha_i) \ge \mbox{ht}(\alpha_{i+1}).
$
Let 
$\gamma_i:=\alpha_i+(k+1)z=0$ for $\alpha_i \in \Phi^+$.
In particular, $\alpha_1$ 
is the highest root. Now, let 
$$
\A_i:=\A \setminus \{H_1,\ldots,H_i\},
$$
where $H_i:=\ker \gamma_i$. 
We would like to apply Theorem \ref{pd1}. For that, what we want to show is the following.

\begin{prop}
In the above notation, for $0 \le i \le n-1$,
$(\A_i)_X$ is free for all $X \in L_2(\A_i^{H_{i+1}})$. 
\label{keyprop}
\end{prop}

\begin{proof}
Let us arrange the notation. Let $\A=\A_i$, $H=H_{i+1}$ and let 
$\A'=\A_{i+1}$. What we have to show is that $\A_X$ is free for all $X \in L_2(\A^H)$. 
By definition, we know that there is $\alpha' \in \Phi_X^+$ such that $H=\{\alpha'+(k+1)z=0\} \in \A_X$. 

Let $\A^\infty $ be the coning of the affine Weyl arrangement, i.e., $\A^\infty $ consists of 
$z=0$ and 
$$
\alpha=t z\ (\alpha \in \Phi^+, t \in \Z).
$$
Then $\A \subset \A^\infty$. Clearly, $\A^\infty|_{z=0} =\A_{\Phi^+}$. 

We distinguish two cases according to the features of $\A_X$. In case (1), $X \not \subset \{z=0\}$, while in case (2) $X \subset \{z=0\}$. 
If we are in case (1), then $\A_X^{z=0}$ is a subset of a parabolic subsystem of $\A_{\Phi^+}$ whose type is either $A_1^3, A_1 \times A_2$ or $A_3$ because $\Phi$ is simply-laced. For the first and second cases, $\A_X$ is free since it is a product of free arrangements. So in case (1), we have to check the freeness of $\A_X$, and this is subset of the parabolic root system of $\A^\infty$ of the form 
$$
a+t_az,b+t_bz,c+t_cz ,a+b+t_{a+b}z,b+c+t_{b+c}z,a+b+c+
t_{a+b+c}z,
$$
where $a,b,c \in \Phi^+$ and $a \mapsto t_a$ is an additive function with values in $\Z$. 

For the case $(2)$, since $z=0$ belongs to $\A_X$, if $\{\alpha+sz=0\} \in \A_X,  $ then $\{\alpha+sz=0\} \in \A_X$ 
for all $-k-1 \le s \le k+2 $ if $\{\alpha+sz=0\} \in \A$. 
So $\A_X^{z=0}$ is a deformation of a Weyl arrangement of type $A_1^2$ or $A_2$. Since the former is free, we may assume that $\A_X$ is a subset of 
$$
z=0,a=s z, b=s z,a+b=s z\qquad (-k-1 \le s \le k+2).
$$

Let us carry out the proof in case (1). Namely, 
$\A_X|_{z=0} $ is a subset of $a,b,c,a+b,b+c,a+b+c$. 
If an arrangement is a subarrangement of the $A_3$-type arrangement, due to Stanley in \cite{St}, it is not free if and only if it is generic, i.e., 
consisting of four planes such that any intersection of two distinct planes are in exactly these two planes. So we can classify 
all of them in terms of root systems, i.e., $\A_X$ is not free if and only if one of the following cases occurs:
\begin{align}
    \label{case 1} \A_X|_{z=0}&=\{b,a+b,b+c,a+b+c\}, \mbox{or} \\
    \label{case 2} \A_X|_{z=0}&=\{a,c,a+b,b+c\}, \mbox{or} \\
    \label{case 3} \A_X|_{z=0}&=\{a,b,c,a+b+c\}.
\end{align}

First, assume that \eqref{case 1} occurs. Suppose that $\{b+(k+1)z=0\} \in \A_X$. Note that all $d+(k+1)z \not \in \A$ for a positive root $d$, if $\mbox{ht}(d)>\mbox{ht}(b), $ and 
$\{d+(k+1)z=0\} \in \A$ if $\mbox{ht}(d)<\mbox{ht}(b)$ by the order of deletion. So the value of $t_{a+b},t_{b+c},t_{a+b+c}$ cannot be $k+1$.
Also, all $\{\alpha+kz=0\} \in \A$ if $-k-2 \le k \le k$. Since $t_{a+b}=t_a+k+1$ but $\{a+t_az=0\} \not \in \A_X$, we have that $t_a<-k-2$. The same proof shows that $t_c<-k-2$. Then $t_{a+b+c}<-k-3$, a contradiciton. 

So $t_b\neq k+1$. Assume that $t_{a+b}=k+1$.
Since $\{b+t_bz=0\} \in \A_X \not \ni \{a+t_a z=0\}$, we know that $t_a>k+1$ and $-k-2\le t_b<0$. Since $t_{a+b+c}<k+1$, we know that $t_c<0$. Since $\{c+t_cz=0\} \not \in \A_X$, we obtain $t_c<-k-2$, which implies that 
$t_{b+c}<-k-2$ so $\{b+c+t_{b+c}z =0\} \not \in \A_X$, a contradiction. The same happens if $t_{b+c}=k+1$. 

So let us assume that $t_{a+b+c}=k+1$. If $t_a>k+1<t_c$, then $t_{a+b+c}>k+1$ since $-k-2 \le t_b \le 0$. Note that $t_a<-k-2>t_c $ cannot occur either. Thus, by the symmetry, we may assume that $t_a>k+1,\ t_c<-k-2$. However, in this case no choice of $-k-2\le t_b \le k+1$ can make both $t_{a+b}$ and $t_{b+c}$ between $-k-2$ and $k+1$. In conclusion, case \eqref{case 1} cannot occur.
\medskip

Next, assume that we are in case \eqref{case 2}. Assume that $t_a=k+1$. Since $t_{a+b}<k+1$, it holds that $ t_b <0 $. Thus $\{b+t_bz=0\} \not \in \A_X$ shows that $t_b<-k-2$. Since $\{b+c+t_{b+c}z=0\} \in 
\A_X$, it holds that $0<t_c\le k+1$ and also we know that $-k-2 \le t_{b+c}<0$. So $-1 \le t_{a+b+c}<k+1$, implying that $\{a+b+c+t_{a+b+c}z=0\} \in \A_X$, a 
contradiction. $t_c=k+1$ cannot occur by symmetry. Next assume that $t_{a+b}=k+1$. If $t_b>k+1$, then $-k-2\le t_c<0$. 
Thus $-1 \le t_{a+b+c} <k+1$ and hence $\{a+b+c+t_{a+b+C}z=0\} \in \A_X$, a contradiction. So assume that $t_b<-k-2$. 
Then $t_{a+b}=t_a+t_b<k+1-k-2=-1$, a contradiction. Thus $t_{a+b}\neq k+1$.
The same holds when $t_{b+c}=k+1$.
\medskip

Finally assume that case \eqref{case 3} occurs. Assume that $t_{a+b+c}=k+1$. 
If all $t_a,t_b,t_c$ are non-negative, then $0 \le t_a+b \le t_{a+b+c}
\le k+1$, so $\{a+b+t_{a+b}z=0\} \in \A_X$, a contradiction. 
If $t_a,t_b$ are non-negative and $t_c<0$, then $k+1<t_{a+b} \le 2k+2$, thus 
$-k-1\le t_c\le -1$. Then $-k-1 \le t_{b+c} \le k+1$, which implies that 
$\{b+c+t_{b+c}z=0\} \in \A_X$, a contradiction. 
The same occurs if $t_b,t_c$ are non-negative. 
Assume that $t_a \ge  0$ and $t_b<0$. Since 
$\{a+t_az=0\}$ and $\{a+t_bz=0\}$
belong to $\A_X$, it is clear that $
\{a+b+t_{a+b}z=0\} \in \A_X$, a contradiction. By the symmetry, 
$t_c\ge 0$ and $t_b <0$ cannot occur. So the rest case is when 
$-k-2\le t_a<0,0\le t_b\le k+1$ and $-k-2\le t_c <0$. Then it is clear that $-k-2\le k_{a+b} <k+1$, so $\{a+b+t_{a+b}z =0\} \in \A_X$, a contradiction.
So $t_{a+b+c} \neq k+1$. 

Assume that $t_a=k+1$. Note that in this case $t_{a+b+c} \neq k+1$.
Since $-k-2 >t_{a+b}$ or $t_{a+b} \ge k+1$, we have
$k+1 \ge t_b \ge 0$.
Since $-k-2 \le t_{a+b+c} \le k$, we also have
$-k-2 \le t_c <0$. Then 
$-k-2 \le t_{b+c}\le k$, so 
$\{b+c+t_{b+c}z=0\} \in \A_X$, again a contradiction. 
Note that $t_c=k+1$ cannot occur by symmetry.
Finally, assume that $t_b=k+1$. Then by the same argument as asbove, it holds that $0\le t_a \le k+1$
and $0 \le t_c \le k+1$. Then $k+1 \le t_{a+b+c}$, and 
$\{a+b+c+t_{a+c+c}z=0\} \not \in \A_X$, which is a contradiction.

So the remaining case is that the ideal of $X$ is generated by $\alpha+(k+1)z,\beta+(k+1)z,z$. In this case $(\A_i)_X$ is the extended Shi-arrangement minus one or two planes of type $A_1^2$ or $A_2$. Since the former is always free, we may assume that $(\A_i)_X$ is of the form 
$$
\mathcal{D}:x-s z=0,\ y-s z=0,\ x+y-s z=0,\ z=0\ (-k-1\le s \le k+2)
$$
form which  $L_1:x+y+(k+1)z=0$, or both 
$L_1:x+y+(k+1)z=0,\ L_2:y+(k+1)z=0$ are deleted.

First, note that  $\mathcal{D}$ is free with exponents 
$(1,3k+6,3k+6)$. Let us compute $|\mathcal{D}^{L_1}|$. To do it, we note that $\mathcal{D}^{L_1}$ consists of
$$
z=0,\ 
y=(-2k-3)z,\ldots,(k+2)z.
$$
So  we have $|\mathcal{D}^{L_1}|=3k+7$ and Theorem \ref{Teraodeletion} shows that $\mathcal{D}_1:=\mathcal{D} \setminus \{L_1\}$ is free with exponents $(1,3k+5,3k+6)$. Next, we delete $L_2$, so we count 
$|\mathcal{D}_1^{L_2}|$. This is given by the hyperplanes 
$$
z=0,\ 
x=(-k-1)z,\ldots,(k+3)z.
$$
So the cardinality is $3k+6$, and Theorem \ref{Teraodeletion} shows that $\mathcal{D}_2:=\mathcal{D}_1 \setminus \{L_2\}$ is free with exponents $(1,3k+5,3k+5)$. 

In conclusion, $(\A_i)_X$ is free for all $X\in L_2(\A_i^{H_{i+1}})$. 
\end{proof}

This proves the following result.

\begin{theorem} \label{Aellpd1}
Let $\Phi$ be a simply laced root system of rank at least three. 
Then $$
\pd_S (D(c\A^{[-k,k+2]}_{\Phi^+}))=1.
$$
\end{theorem}

\begin{proof}
    We combine Proposition \ref{keyprop} with Theorems \ref{SPOG} and  \ref{pd1}.
So what we have to show is that $\A=c\A_{\Phi^+}^{[-k.k+2]}$ is not free, which we will show in the following. Let $X \in L_3(\A)$ be defined by 
$z=\alpha=\beta=0$, where $\alpha$ and $\beta$ are part of the simple roots of $\Phi^+$ such that $\alpha+\beta \in \Phi^+$. Then $\A_X$ is nothing but $c\A_{A_2^+}^{[-k,k+2]}$ which is not free by Theorem \ref{theoremAFV}. So $\A$ cannot be free by Proposition \ref{local}. \end{proof} 

\section{Linear resolution in type $A_3$}

\label{section:A3}

In this section, we compute explicitly the minimal free 
resolution of the deformation $c\A^{[-k,k+2]}_{\Phi^+}$ of the Weyl arrangement of the type $A_3$, for $k \ge 0$.
We rely on Theorem \ref{multiple-deletion}. 
The goal is to prove Theorem \ref{A3case}, which we restate as follows.

\begin{theorem} \label{k2case}
Let $\A_3$ be the braid arrangement in $\K^{4}$ and let 
$\A^{[-k,k+2]}_{\Phi^+}$ be its deformation:
$$
Q(\A^{[-k,k+2]}_{\Phi^+})=\prod_{1 \le i < j \le 4}\prod_{-k \le c \le k+2} (x_i-x_j-c).
$$
Then $D_0(c\A^{[-k,k+2]}_{\Phi^+})$ has the following free resolution:
$$
0
\rightarrow 
S[-4k-8]^3
\rightarrow 
S[-4k-7]^6
\rightarrow 
D_0(c\A^{[-k,k+2]}_{\Phi^+}) \rightarrow 0.
$$
\end{theorem}

The strategy of the proof is as follows. We start from the Shi arrangement which is free as in Theorem \ref{thm:CatShi}, and delete six hyperplanes from it in a certain order. In the first three steps of deletion, we can keep the freeness by using several deletion theorems like Theorems \ref{MDT}, \ref{MDT2} 
and \ref{Ycriterion}. In the last deletion of three hyperplanes, we apply Theorem \ref{SPOGMDT} to complete the proof. Now let us start the proof.

\begin{proof}[Proof of Theorem \ref{k2case}]
    Let $\A:=c\A^{[-k-1,k+2]}_{\Phi^+}$. This is the 
extended Shi arrangemenet, whose exponents are $(1,4k+8,4k+8,4k+8)$ by Theorem \ref{thm:CatShi}. First, let us compute 
$|\A|-|\A^H|$ for $\alpha=(-k-1)z$ for non-simple roots $\alpha$. 

Let us first analyze the case $H:x_1-x_4=(-k-1)z$. We have 
\begin{align*}
Q(\A^H)=
z&\prod_{-2k-3 \le c \le k+2} (x_2-x_4-cz)(x_3-x_4-cz) \cdot \\ 
&\prod_{-k-1 \le c \le k+2}(x_2-x_3-cz).
\end{align*}
Thus, $|\A^H|=8k+17$ and $|\A|-|\A^H|=4k+8 \in \exp(\A)$. Hence Theorem \ref{MDT} shows that 
$\A_0:=\A \setminus \{H\}$ is free with $\exp(\A_0)=(1,4k+7,4k+8,4k+8)$. 
\bigskip

Next, let us look at $H:x_1-x_3=(-k-1)z$. This time we have 
\begin{align*}
    Q(\A_0^H)=
z&\prod_{-2k-3 \le c \le k+2} (x_2-x_3-cz)\prod_{-k-1 \le c \le 2k+3} (x_3-x_4-cz)\\
&\prod_{-k-1 \le c \le k+2}(x_2-x_4-cz).
\end{align*}
So, we have $|\A_0^H|=8k+16$ and $|\A_0|-|\A^H|=4k+8 \in \exp(\A)$. 
Hence, Theorem \ref{MDT2} shows that 
$\A_1:=\A_0 \setminus \{H\}$ is free with $\exp(\A_1)=(1,4k+7,4k+7,4k+8)$. 
\bigskip

Further, let us consider $H:x_2-x_4=(-k-1)z$. In this case we have 
\begin{align*}
Q(\A_1^H)=
z&\prod_{-2k-2 \le c \le k+2} (x_1-x_4-cz)
\prod_{-2k-3 \le c \le k+2} (x_3-x_4-cz)\\
&\prod_{-k \le c \le k+2}(x_1-x_3-cz).
\end{align*}
Hence, $|\A_1^H|=8k+15$ and $|\A_1|-|\A_1^H|=4k+8 \in \exp(\A)$. 
We see that, for $\A_2:=\A_1 \setminus \{H\}$, $b_2^0(\A_2^H):=
b_2(\A_2^H)-|\A_2^H|+1=(4k+7)^2$, and $\exp((\A_2^H)^{z=0},m^{z=0})=(4k+7,4k+7)$. Hence, 
Theorem \ref{Ycriterion} shows 
that $\A_2^{H}$ is free with $\exp(\A_2^H)=(1,4k+7,4k+7)$. 
So Theorem \ref{Teraodeletion} shows that 
$\A_2$ is free with $\exp(\A_1)=(1,4k+7,4k+7,4k+7)$. 
\bigskip

Finally, for $H_1:=\{x_1-x_2=-k-1\}$, 
$H_2:=\{x_2-x_3=-k-1\}$, and 
$H_3:=\{x_3-x_4=-k-1\}$, we can directly compute that $|\A_3^{H_i}|=8k+14$ and $|\A_2|-
|\A_2^{H_i}|=4k+8$. Since $H_1,H_2,H_3$ satisfty 
the conditions in Theorem \ref{SPOGMDT}, we see that $\A^{[-k,k+2]}_{\Phi^+}=
\A_2 \setminus \{H_1,H_2,H_3\}$ has the following minimal free resolution
$$
0 \rightarrow 
S[-4k-8]^3 
\rightarrow 
S[-4k-7]^6 
\rightarrow 
D(\A^{[-k,k+2]}_{\Phi^+}) 
\rightarrow 
0.
$$
This completes the proof. \end{proof}

\section{Arrangements of type $B_2$}
\label{section:B2}

In this section, we prove Theorem \ref{resolution-B2}. Let $\Phi$ be the root system of type $B_2$, thus 
$$
\A_{\Phi^+}=\{xy(x^2-y^2)=0\}
$$
This is an arrangement in $V=\K^2$. We focus on its deformations, which are the arrangements in $\K^3$ defined, for given integers $k, j \ge 0$, as 
$$
\A_{\Phi^+}^{[-k,k+j]}= \bigcup_{-k \le s \le k+j} \{x=s\} \cup \{y=s\} \cup \{x+y=s\} \cup \{x-y=s\}.
$$

To prove Theorem \ref{resolution-B2}, we start with the following result.
Write $j=2m+r$, $r \in \{0,1\}$.

\begin{prop} \label{CatShi}
Let $\A=c\A_{\Phi^+}^{[-k,k+2m+r]}$, with $r \in \{0,1\}$, and let 
$$
\B=\left(\bigcup_{-k-m \le s \le -k-1} \{y=sz\}  \cup \A \right)\setminus \bigcup_{k+m+r \le s \le k+2m+r} \{y=sz\}.
$$
Then $\B$ is free with $\exp(\B)=(1,4k+4m+1+3r,4k+4m+3+r)$.
\end{prop}

\begin{proof} 
    Replace $x$ by $x+kz$. Then $\B$ is a Catalan arrangement for $r=0$, while $\B$ is a Shi arrangement for $r=1$. 
     This completes the proof, since the Catalan and the Shi  arrangements are free with the desired exponents by Theorem \ref{thm:CatShi}.
\end{proof}

\subsection{Addition and deletion for arrangements of type $B_2$}

The goal of this subsection is to complete the proof of 
Theorem \ref{resolution-B2}. Following the strategy suggested by Proposition \ref{CatShi}, we add and delete lines successively and we keep track of the behavior of the minimal graded free resolution at each step to reach the final statement. We develop the proof a bit differently, depending on whether $j$ is even or odd, i.e. whether   $r=0$ or $r=1$. 

\begin{proof}[Proof of Theorem \ref{resolution-B2}] First 
let us assume that $r=0$. Then we know that the arrangement 
$\B$ in Proposition \ref{CatShi} 
is free with exponents $(1,t+1,t+3)$, where $t=4k+4m$. 
To pass from $\B$ to $\A$ we add/delete hyperplanes iteratively and, at each step,
we apply Lemmas \ref{lemma1} and \ref{lemma2}.
More precisely, we need to delete $H:y=(-k-m+u)z$ and add 
$L_u:y=(k+m+u+1)z$, for $u \in \{0,\ldots,m-1\}$.
\medskip

For all integers $u$ with $0 \le u \le m$, we define the following arrangement:
$$
\B_u:=\left(\bigcup_{-k-m+u \le s \le -k-1} \{y=sz\} \cup \A \right)\setminus \bigcup_{k+m+u+1 \le s \le k+2m} \{y=sz\},
$$
so that $\B_0=\B$ and $\B_m=\A$. 
We prove that, for $u\ge 0$,
there is a minimal free resolution
\begin{eqnarray*}
0 &\rightarrow& 
S[-t-3-u]\oplus \bigoplus_{i=2}^u S[-t-2-u-i]^2\\
&\rightarrow &
S[-t-1-u]\oplus \bigoplus_{i=1}^u S[-t-1-u-i]^2
\rightarrow D_0(\B_u) \rightarrow 0
\end{eqnarray*}

First, let us look at the case $u=0$.
In order to reach $\B_1$ starting from $\B_{0}=\B$, we need to delete $H_0:y=(-k-m)z$ and add 
$L_0:y=(k+m+1)z$.  
Let us look at the effect of deleting $H_0$. Recall that $\B_0=\B$ consists of hyperplanes parallel to $y=0$, $z=0$ and 
$$
x,x\pm y=-k,\ldots,k+2m.
$$
Hence, looking at the intersection points with  $H_0$, we see that
$
\B^{H_0}
$ consists of $z=0$ and $
y=2k+3m,\ldots,-2k-m$.
So $|\B^{H_0}|=4k+4m+2$ and therefore $\exp(\B^{H_0})=(1,t+1)$. Hence, Theorem \ref{Teraodeletion} shows that 
$\B \setminus \{{H_0}\}$ is free with exponents $(1,t+1,t+2)$. 
Now let us add $L_0$ to $\B \setminus \{{H_0}\}$, which is $\B_{1}$. Note that $\B_{1}^{L_0}$ consists of $z=0$ and 
$$
y=2k+3m+1,\ldots,-2k-m-1.
$$
So $|\B_{u+1}^{L_0}|=t+4$, so $d$ in Lemma \ref{lemma2} is $t+3$, which satisfies the conditions in Lemma \ref{lemma2}. Thus, Lemma \ref{lemma2} shows that 
we have a minimal free resolution 
\[
0 \rightarrow 
S[-t-4]
\rightarrow 
S[-t-2]\oplus  S[-t-3]^2 
\rightarrow D_0(\B_{1}) \rightarrow 0,
\]
which matches the resolution above when $u=0$. 
\bigskip

Now we look at the induction step. Assume that the statement is true 
up to $u \ge 0$. Starting from the arrangement $\B_u$, in order to obtain $\B_{u+1}$, we need to delete ${H_u}:y=(-k-m+u)z$ and add 
$L_u:y=(k+m+u+1)z$.  
First, let us proceed with deletion of ${H_u}$. Recall that $\B_u$ consists of hyperplanes parallel to $y=0$, $z=0$ and 
$$
x,x\pm y=-k,\ldots,k+2m.
$$
Hence, we see that the restricted arrangement
$
\B_u^{H_u}
$ consists of $z=0$ and $
y=2k+3m-u,\ldots,-2k-m+u$.
So $|\B_u^{H_u}|=4k+4m-2u+2$. Since $|\B_u|=8k+8m+5$, we have that 
$d:=|\B_u|-1-|\B_u^{H_u}|=t+2u+2$, which satisfies the condition in Lemma \ref{lemma1}. So we have a minimal free resolution 
\begin{eqnarray*}
0 &\rightarrow& 
S[-t-3-u]\oplus \bigoplus_{i=2}^u S[-t-2-u-i]^2 \oplus S[-t-2u-3]\\
&\rightarrow &
S[-t-1-u]\oplus \bigoplus_{i=1}^u S[-t-1-u-i]^2 \oplus S[-t-2u-2]
\rightarrow D_0(\B_u \setminus \{{H_u}\}) \rightarrow 0.
\end{eqnarray*}

Now let analyze the arrangement $\B_{u+1}$, which is obtained by adding $L_u$ to $\B_u \setminus \{{H_u}\}$. Note that $\B_{u+1}^{L_u}$ consists of $z=0$ and 
$$
y=2k+3m+u+1,\ldots,-2k-m-u-1.
$$
So $|\B_{u+1}^{L_u}|=t+2u+4$, so $d$ in Lemma \ref{lemma2} is $t+2u+3$, which satsifies the conditions in Lemma \ref{lemma2}. Thus Lemma \ref{lemma2} shows that 
we have a minimal free resolution 
\begin{eqnarray*}
0 &\rightarrow& 
S[-t-4-u]\oplus \bigoplus_{i=2}^{u+1} S[-t-3-u-i]^2 \\
&\rightarrow &
S[-t-2-u]\oplus \bigoplus_{i=1}^{u+1} S[-t-2-u-i]^2 
\rightarrow D_0(\B_{u+1}) \rightarrow 0.
\end{eqnarray*}
This completes the proof in the case $r=0$.

\bigskip

So let us assume that $r=1$. Then we know that the arrangement 
$\B$ defined in Proposition \ref{CatShi} in the case $r=1$ is free with exponents $(1,t+4,t+4)$ by Theorem \ref{thm:CatShi}, where $t=4k+4m$. To pass from $\B$ to $\A$, again we add/delete hyperplanes in a recursive manner, applying at each step Lemmas \ref{lemma1} and \ref{lemma2}.
Thus, for all integers $u$ with $0 \le u \le m$, we set
$$
\B_u:=
\left(\bigcup_{-k-m+u \le s \le -k-1}\{y=sz\} \cup \A \right)\setminus \bigcup_{k+m+u+1 \le s \le k+2m+1} \{y=sz\},
$$
so again $\B_0=\B$ and $\B_m=\A$. We prove that, for $0 \le u \le m$,
there is a minimal free resolution
\[
0 \rightarrow 
\bigoplus_{i=1}^{u} S[-t-5-u-i]^2
\rightarrow 
\bigoplus_{i=0}^{u} S[-t-4-u-i]^2
\rightarrow D_0(\B_u) \rightarrow 0.
\]

The case $u=0$ is nothing but Proposition \ref{thm:CatShi}. Looking at the induction step, we assume that the statement is true 
up to $u \ge 0$. To convert $\B_u$ into $\B_{u+1}$, we need to delete ${H_u}:y=(-k-m+u)z$ and add 
$L_u:y=(k+m+u+2)z$.  
Let us delete ${H_u}$. Recall that $\B_u$ consists of hyperplanes parallel to $y=0$, $z=0$ and 
$$
x,x\pm y=-k,\ldots,k+2m+1.
$$
Hence, we see that 
$
\B_u^{H_u}
$ consists of $z=0$ and $
y=2k+3m-u+1,\ldots,-2k-m+u$.
So $|\B_u^{H_u}|=4k+4m-2u+3$. Since $|\B_u|=8k+8m+9$, we get
$d:=|\A|-1-|\A^{H_u}|=t+2u+5$, which satisfies the condition in Lemma \ref{lemma1}. So we have a minimal free resolution 
\begin{eqnarray*}
0 &\rightarrow& 
\bigoplus_{i=1}^{u} S[-t-5-u-i]^2 \oplus S[-t-2u-6]\\
&\rightarrow &
\bigoplus_{i=0}^{u} S[-t-4-u-i]^2 \oplus S[-t-2u-5]
\rightarrow D_0(\B_u \setminus \{{H_u}\}) \rightarrow 0.
\end{eqnarray*}

Finally, let us add $L_u$ to $\B_u \setminus \{{H_u}\}$, whereby obtaining the arrangement $\B_{u+1}$. Note that $\B_{u+1}^{L_u}$ consists of $z=0$ and 
$$
y=2k+3m+u+3,\ldots,-2k-m-u-2.
$$
So $|\B_{u+1}^{L_u}|=t+2u+7$, so the value of the quantity $d$ appearing in Lemma \ref{lemma2} is $t+2u+6$, which satisfies the conditions of Lemma \ref{lemma2}. Therefore, Lemma \ref{lemma2} provides a minimal free resolution 
\[
0 \rightarrow
\bigoplus_{i=1}^{u+1} S[-t-6-u-i]^2 \rightarrow 
\bigoplus_{i=0}^{u+1} S[-t-5-u-i]^2 
\rightarrow D_0(\B_{u+1}) \rightarrow 0.
\]
This completes the proof when $r=1$.
\end{proof}

\subsection{Stability and jumping lines}

Let $\Phi^+$ be set of positive roots of the root system of type $B_2$ and consider $\A:=c\A^{[-k,k+j]}_{\Phi^+}$.
Put $\E = \widetilde{D_0(c\A)}$.
In this subsection, we study the stability properties of the sheaf $\E$, which is a vector bundle of rank two on $\P^2=\P^2_\K$.

\begin{lemma}
Assume $j \ge 3$.
    Then $\E$ is a slope-stable bundle.
\end{lemma}

\begin{proof}
Write $j=2m+r$ for some integers $(m,r)$ with $r \in \{0,1\}$.
Note that the arrangement $\A$ consists of $4(2k+j+1)+1$ lines in $\P^2$, so that the rank-$2$ vector bundle $\E$ satisfies
\[
c_1(\E)=-4(2k+j+1)h,
\]
where $h$ is the hyperplane divisor of $\P^2$.

We consider $\F=\E(2(2k+j+1))$, so that $c_1(\F)=0$.
In view of Hoppe's criterion (see \cite{hoppe:generischer}), to prove that $\E$ is stable it suffices to check that $H^0(\F)=0$.
However, sheafifying the minimal graded free resolution provided by Theorem \ref{resolution-B2}, we see that that $H^0(\E(i h))=0$, for all integers $i < 1+5m+3r+4k$.
Then, we have to check that $2(2k+j+1) < 1+5m+3r+4k$, which means $m + r > 1$. This is equivalent to $j \ge 3$, so the lemma is proved.
\end{proof}

\begin{rem}
    With the same notation as the previous lemma, we note that
    for $j=2$, the bundle $\E$ is slope-semistable (but not stable). Indeed, sheafifying the resolution of Theorem \ref{main}, we get the exact sequence
    \[
    0 \to \mathcal{O}_{\P^2} \to \F \to \mathcal{I}_{p/\P^2} \to 0,
    \]
    where $p$ is a point in $\P^2$. In can be seen from Theorem \ref{B2:jump} below, that $p$ is the intersection point of the lines $H_{m-1}$ and $L_{m-1}$, which it to say that $p=(1:0:0)$.
\end{rem}

\begin{define}
Given a line $\ell \subset \P^2$, there is a non-negative integer $b=b_\ell$ such that
\[
\F|_\ell \simeq \cO_\ell(b_\ell) \oplus \cO_\ell(-b_\ell).
\]
We say that the \textbf{jumping order of $\ell$ for $\E$ (or for $\F$) is $a$}, or that $\ell$ is a \textbf{$a$-jumping} line of $\E$ (or of $\F$) if $b_\ell=a > 0$. 
This notion does not depend on twisting $\E$ by a line bundle.
\end{define}
Since $\E$ is slope-semistable for $j\ge 2$, we have 
$b_\ell=0$ when $\ell$ is a sufficiently general line in $\P^2$, so
$\ell$ is not jumping.
From the proof of Theorem \ref{resolution-B2}, we deduce the following result.

\begin{theorem} \label{B2:jump}
    For $u \in \{1,\ldots,m-1\}$, consider the lines 
    \[H_u:y=(-k-m+u)z, \qquad 
L_u:y=(k+m+u+r+1)z.\]
Then $H_u$ and $L_u$ are $(2u+r+1)$-jumping lines of $\E$. 
In addition, for $m\ge 2$, $H_{m-1}$ and $L_{m-1}$ are the only $(j-1)$-jumping lines of $\E$ and this is the maximal jumping order.
\end{theorem}

\begin{proof}
We work with $\A=\A_{\Phi^+}^{[-k,k+2m+r]}$ and the vector bundle $\E = \widetilde{D_0(c\A)}$ of rank $2$ on $\P^2$.
  Set $\cL_u=\cO_{\P^2}(-m-r-u)$.
  Sheafifying the minimal graded free resolution obtained in Theorem
  \ref{resolution-B2}, we get a short exact sequence of locally free sheaves:
  \begin{equation} \label{resolution-B2-sheaf}
      0 \to 
    \begin{array}{c}
      \cL_0(-1)^{\oplus (1+r)} \\
      \oplus \\
      \displaystyle{\bigoplus_{1 \le u \le m-1}} \cL_u(-1)^{\oplus 2}
    \end{array} \to
    \begin{array}{c}
      \cL_0(1)^{\oplus(1+r)} \\
      \oplus \\
      \displaystyle{\bigoplus_{0 \le u\le m-1}}
      \cL_u^{\oplus 2}
    \end{array} \to 
     \F \to 0 
    \end{equation} 
  For $1 \le u \le m-1$, let us consider the $2\times 2$ matrix of linear forms $M_u$
  extracted from this resolution, taking the form:
  \[
    M_u : \cL_u(-1)^{\oplus 2} \to \cL_u^{\oplus 2}.
  \]
  By construction, $M_u$ is obtained from the $u$-th step of the
  addition/deletion procedure in the proof of Theorem
  \ref{resolution-B2} and, as such, $M_u$ is the diagonal matrix given
  by the equations defining $L_u$ and $H_u$. Set $C_u=\mathrm{coker}(M_u)$.
  \medskip

  Let us show that, for $1 \le u \le m-1$, the lines $H_u:y=(-k-m+u)z$ and $L_u:y=(k+m+u+r+1)z$ are 
$(2u+r+1)$-jumping lines for $\E$. 
We use Theorem \ref{AFV}. We compute the following values:
$$
|\A^{H_u}|=4k+4m+r-2u+2, \qquad 
|\A \cap L_u|=4k+4m+3r+2u+4.
$$
Since $|\A|=1+4(2k+2m+r)=8k+8m+4r+1$, 
the conditions in Theorem \ref{AFV} are satisfied for both $H_u$ and $L_u$. 
Recall that, by definition, 
\[\F=\E(4k+2j+2)) = \E(4k+4m+2r+2).\]
Hence, we get the following splitting
\begin{eqnarray*}
\F|_{H_u} 
&\simeq& \cO_{H_u}(-r-2u-1)\oplus 
\cO_{H_u}(-r+2u-1),\\
\F|_{L_u} 
&\simeq& \cO_{L_u}(-r-2u-1)\oplus 
\cO_{L_u}(-r+2u-1).
\end{eqnarray*}
So $L_u$ and $H_u$ are $(2u+r+1)$-jumping lines for $\F$, equivalently, for $\E$. In particular, 
$H_{m-1}$ and $L_{m-1}$ are 
$(j-1)$-jumping lines. \medskip

Now, we claim that there is a splitting:
\begin{equation} \label{splitting}
\mathrm{coker}(M_{m-1}) \simeq \cO_{L_{m-1}}(1-j) \oplus 
\cO_{H_{m-1}}(1-j).
\end{equation}

Indeed, first we note that $\mathrm{coker}(M_{m-1})$ is an extension of the sheaves
$\cO_{L_{m-1}}(1-j)$ and $\cO_{H_{m-1}}(1-j)$. 
To see this, first we use that $H_{m-1}$ and $L_{m-1}$ are $(j-1)$-jumping lines of $\F$.
So, in view of the resolution \eqref{resolution-B2-sheaf}, we have thus a surjection map $\cL_u^{\oplus 2}|_{H} \to \cO_H(1-j)$, with $0 \le u \le m-1$ or $\cL_0(1)^{\oplus (1+r)}|_{H} \to \cO_H(1-j)$. By the definition of $\cL_u$, this implies $-m-r-u \le 1-j$, i.e. $u \ge m-1$, so $u = m-1$.
This forces $u = m-1$. Looking at the resolution \eqref{resolution-B2-sheaf}, we deduce that $\cO_{L_{m-1}}(1-j)$ and $\cO_{H_{m-1}}(1-j)$ are both quotients of $\mathrm{coker}(M_{m-1})$. 
Therefore, 
$\mathrm{coker}(M_{m-1})$ is a splitting extension of 
$\cO_{L_{m-1}}(1-j)$ and $\cO_{H_{m-1}}(1-j)$ and \eqref{splitting} is proved.

\medskip
   
Let us check that $L_{m-1}$ and $H_{m-1}$ are the only $(j-1)$-jumping lines of $\E$. Let $H$ be a $(j-1)$-jumping line of $\E$, or equivalently of $\F$. Then, again we have a surjection $\F\to \cO_H(1-j)$.
Once more, we get that $\cO_H(1-j)$ is a quotient of $\mathrm{coker}(M_{m-1})$.
Therefore, in view of \eqref{splitting}, we conclude that $H$ is either $L_{m-1}$ or $H_{m-1}$.
\medskip

By the same argument, if $H$ is a line whose jumping order is $b$ with $b > 2m+r-1$, then we should have a nonzero map $\cL_u^{\oplus 2}|_{H} \to \cO_H(-b)$, with $0 \le u \le m-1$, or $\cL_0(1)^{\oplus (1+r)}|_{H} \to \cO_H(-b)$. But this would imply $-m-r-u \le - b$, i.e. $u \ge b-m-r > m-1$. However, by definition $u \le m-1$, so this is impossible. Hence $2m+r-1$ is the maximal jumping order of a line $H \subset \P^2$ for the vector bundle $\E$.
\end{proof}

\begin{cor} \label{not-isomorphic}
    Let $j \ge 3$ be an integer, consider $k,k'$ distinct non-negative integers and set $\A = c\A^{[-k,k+j]}_{B_2}$, $\A' = c\A^{[-k',k'+j]}_{B_2}$. Then $D_0(\A)$ and $D_0(\A')[4(k'-k)]$     share the same graded Betti numbers but are not isomorphic.
\end{cor}

\begin{proof}
    It follows from Theorem \ref{resolution-B2} that the graded Betti numbers of $D(\A)$ and $D(\A')$ only differ by a degree shift and that such a shift is precisely $4(k-k')$. 
    So $D_0(\A)$ and $D_0(\A')[4(k'-k)]$ have the same graded Betti numbers.

    If these modules were isomorphic, then $D_0(\A)$ and $D_0(\A')[4(k'-k)]$ would give, after sheafification, two isomorphic vector bundles $\E$ and $\E'$ on $\P^2$. So the lines of any given jumping order of these bundles would coincide. However, this is a contradiction. Indeed, the $(j-1)$-jumping lines of  $\E$ are $L:\{y=(k+j)z\}$ and $H:\{y=(-k-1)z\}$, while those of $\E'$ are 
    $L':\{y=(k'+j)z\}$ and $H':\{y=(-k'-1)z\}$ and these are different lines if $k \ne k'$.
\end{proof}

It is not clear to us whether $\E$ and $g^*(\E')$ could be isomorphic, 
where $g \in \mathrm{PGL_3(}(\K)$ sends $L'$ to $L$ and $H'$ to $H$.

\bibliographystyle{amsalphav2}
\bibliography{arrangements.bib}

\bigskip
\noindent \textsc{Takuro Abe}. 
Rikkyo University, 3-34-1 Nishi Ikebukuro, 
Toshima-ku, 1718501 Tokyo, Japan.
Email \texttt{abetaku@rikkyo.ac.jp}

\bigskip
\noindent \textsc{Daniele Faenzi}. 
Université Bourgogne Europe, CNRS, IMB UMR 5584, F-21000 Dijon, France. Email \texttt{daniele.faenzi@ube.fr}.

\end{document}